\documentclass[12pt]{amsart}
\usepackage{amsmath}
\usepackage{enumerate}
\usepackage{amssymb}
\usepackage{amsbsy}
\usepackage{amsfonts}
\usepackage[arrow]{xy}
\usepackage{comment}

\theoremstyle{plain}
\newtheorem{thm}{Theorem}[section]
\newtheorem{prop}[thm]{Proposition}
\newtheorem{lem}[thm]{Lemma}
\newtheorem{cor}[thm]{Corollary}
\theoremstyle{definition}

\numberwithin{equation}{section}
\newcommand{\sm}{\left(\begin{smallmatrix}}
\newcommand{\esm}{\end{smallmatrix}\right)}

\newfont{\FieldFont}{msbm10 scaled\magstep1}

\def\SL{\mathrm{SL}_2(\mathbb Z)}

\usepackage{color}

\definecolor{blue}{rgb}{0,0,1}
\definecolor{red}{rgb}{1,0,0}
\definecolor{green}{rgb}{0,.6,.2}
\definecolor{purple}{rgb}{1,0,1}

\long\def\red#1\endred{{\color{red}#1}}
\long\def\blue#1\endblue{{\color{blue}#1}}
\long\def\purple#1\endpurple{{\color{purple}#1}}
\long\def\green#1\endgreen{{\color{green}#1}}

\makeatletter
\newcommand\matc[4]{\left( {#1\@@atop #3}{#2\@@atop #4}\right)}
\newcommand\matr[4]{\left( {\hfill #1\@@atop\hfill #3}{\hfill
#2\@@atop\hfill #4}\right)}
\makeatother

\begin{document}

\title[On Rankin-Cohen Brackets of Hecke Eigenforms]{On Rankin-Cohen Brackets of Hecke Eigenforms and Modular Forms of Half-Integral Weight}
\author{YoungJu Choie}
\address{Department of Mathematics, Pohang University of Science and Technology, Pohang, Republic of Korea}\email{yjc@postech.ac.kr}\author{Winfried Kohnen}\address{Mathematisches Institut der Universität INF 205, D-69120, Heidelberg, Germany} \email{winfried@mathi.uni-heidelberg.de} \author{Yichao Zhang}\address{MOE-LCSM, School of Mathematics and Statistics, Hunan Normal University, Changsha, Hunan 410081, P. R. China}\email{yichao.zhang@hunnu.edu.cn}

\thanks{The first author was partially supported by   NRF 2022R1A2B5B0100187113, BSRI-NRF 2021R1A6A1A10042944, and third author was partially supported by NSFC 12271123.}
 

\date{}
\maketitle

\begin{abstract}
We generalize the linear relation formula between the square of normalized Hecke eigenforms of weight $k$ and normalized Hecke eigenforms of weight $2k$, to Rankin-Cohen brackets of general degree. As an ingredient of the proof, we also generalize a formula of Zagier on the Petersson inner product of Rankin-Cohen brackets involving Eisenstein series.
\end{abstract}


\section{Introduction}

Let $S_k$ be the space of cusp forms of weight $k$ and level $1.$  In \cite{K}, it was given an explicit description, in terms of special values of certain shifted  $L$-functions and Fourier coefficients of half-integral weight modular forms, of the subspace of $S_{2k}$  generated
by the squares $f(\tau)^2$ of normalized Hecke eigenforms $f(\tau)$  in $S_k.$  The main ingredients in the proof except for the classical Petersson formulas
for Fourier coefficients was a well-known identity relating the first
Shimura lift $\mathcal S_1$ of the product $f(4\tau)\theta(\tau) $ to $  f(\tau)^2,$ where
$\theta(\tau) $ is the standard theta function of weight $\frac{1}{2}$ and level $4$. 
\medskip

The purpose of the present paper is to generalize the above result to
the case where the square $f(\tau)^2$ is replaced
by the Rankin-Cohen bracket $ [f(\tau),f(\tau)]_\nu (\nu\geq 0).$   Besides
the Petersson formulas again, the main
tool in the proof is a generalization due to Popa \cite{P} of the above-mentioned
result referring to $\mathcal S_1 $ as well as a generalization of a formula due to
Zagier \cite{Z}, which expresses the Petersson scalar product of the Rankin-Cohen bracket
against an Eisenstein series in terms of the special value of an $L$-function, to the case where the Eisenstein series is replaced by a Poincar\'e series.

\section{Preliminaries and Results}

Denote the two standard generators of $\Gamma(1)=\SL$ by
\[S=\begin{pmatrix}
0&-1\\
1&0
\end{pmatrix},\quad T=\begin{pmatrix}
1&1\\
0&1
\end{pmatrix}.\]
For a congruence subgroup $\Gamma$ of $\Gamma(1)$, we assume that $-I\in\Gamma$ throughout. Let $\Gamma_\infty$ denote the subgroup of unipotent elements in $\Gamma$ and $w$ the width of the cusp $\infty$ for $\Gamma$, so $\Gamma_\infty$ is generated by $T^w$.

Let $\mathrm{GL}_2^+(\mathbb R)$ be the set of matrices in $\mathrm{GL}_2(\mathbb R)$ of positive determinant. For a real number $k$, the slash-$k$ operator of $\gamma\in \mathrm{GL}_2^+(\mathbb R)$ on functions on the upper-half plane $\mathbb H$ is given by
\[f|_k\gamma(\tau)=\det(\gamma)^{\frac{k}{2}}j(\gamma,\tau)^{-k}f(\gamma\tau),\quad \tau\in\mathbb H.\]
Here $j(\gamma,\tau)=c\tau+d$ for $\gamma=\begin{pmatrix}
a&b\\
c&d
\end{pmatrix}$ and $z^s=e^{s\log z}$ for $z\in\mathbb C^\times$, $s\in\mathbb C$ with $\log z=\log |z|+i\arg(z)$ taking the principal branch (that is, $\arg(z)\in(-\pi ,\pi]$). 

For a multiplier system $v$ of weight $k$ on $\Gamma$, a modular form of weight $k$ and multiplier system $v$ for $\Gamma$ is a holomorphic function $f$ on $\mathbb H$ such that $f|_k\gamma=v(\gamma)f$ for $\gamma\in\Gamma$ and $f$ is holomorphic at cusps. If in addition $f$ vanishes at cusps, $f$ is called a cusp form. The spaces of modular forms and cusp forms are denoted by $M_k(\Gamma,v)$ and $S_k(\Gamma,v)$ respectively. (See Chapter 3 and 4 of \cite{R} for details on multiplier systems and modular forms.) We employ the normalized Petersson inner product on the space $S_{k}(\Gamma,v)$ of cusp forms of weight $k$ and multiplier system $v$ on $\Gamma$ throughout:
\[\langle f,g\rangle=\frac{1}{[\Gamma(1)\colon \Gamma]}\int_{\Gamma\backslash\mathbb H}f(\tau)\overline{g(\tau)}y^k\frac{dxdy}{y^2}.\]
In particular, the inner product is independent of the choice of the group $\Gamma$. 

Let $k>2$ and $v(T^w)=e^{2\pi i\delta}$ with $\delta\in [0,1)$ and $m\in \mathbb Z+\delta$ be a positive real number. Define the $m$-th Poincar\'e series of weight $k$ with multiplier system for $\Gamma$ by
\[P_{k,m,v}(\tau)=\frac{1}{2}\sum_{\gamma\in\Gamma_\infty\backslash\Gamma}v(\gamma)^{-1}(c\tau+d)^{-k}e^{2\pi im\gamma\tau/w}=\frac{1}{2}\sum_{\gamma\in\Gamma_\infty\backslash\Gamma}e^{2\pi im\tau/w}|_{k,v}\gamma(\tau).\]
Hereafter the lower row of $\gamma$ is typically denoted by $(c,d)$. By the choice of $m$, the function $v(\gamma)^{-1}(c\tau+d)^{-k}e^{2\pi im\gamma\tau/w}$ is invariant under left multiplication by $T^w$ on $\gamma$. Moreover, since $k>2$, the series is absolutely convergent and vanishes at $\infty$ and other cusps and it belongs to  $S_{k}(\Gamma,v)$. As usual, by the folding-unfolding trick, with respect to the normalized Petersson inner product
\[\langle h, P_{k,m,v}\rangle=\frac{1}{[\Gamma(1):\Gamma]}\frac{\Gamma(k-1)w^{k}}{(4\pi m)^{k-1}}c(m),\quad h=\sum_nc(n)e^{2\pi in\tau/w}\in S_{k}(\Gamma,v).\]
We denote by $P_{k,m}$ the Poincar\'e series of weight $k$ and degree $m$ for $\Gamma=\SL$ and by $P_{k,m,4}$ that for $\Gamma=\Gamma_0(4)$ with the implicit multiplier system $\gamma\mapsto\theta(\gamma\tau)/\theta(\tau)$.

For each $\nu\in \mathbb{Z}_{\geq 0}$ and modular forms $f_i\in M_{k_i}(\Gamma,v_i)$, $i=1,2$, the Rankin-Cohen brackets are defined as
$$[f_1, f_2]_{\nu}=
(2\pi i )^{-\nu}\sum_{i=0}^{\nu}(-1)^{\nu-i}\sm \nu\\ i\esm 
\frac{\Gamma(k_1+\nu)\Gamma(k_2+\nu)}{\Gamma(k_1+i)\Gamma(k_2+\nu-i)}f_1^{(i)} f_2^{(\nu-i)}.$$
Then $[f_1,f_2]_\nu\in M_{k_1+k_2+2\nu}(\Gamma,v_1v_2)$ (See Theorem 7.1 of \cite{C}).

When $k\in\mathbb Z$ and $v$ is trivial, we drop $v$ from the notation. Let $\theta(\tau)=\sum_{n\in\mathbb Z}q^{n^2}$, $q=e^{2\pi i\tau}$, be Jacobi's theta function. When $k\in\frac{1}{2}+\mathbb Z$ and $\Gamma=\Gamma_0(4)$, the multiplier system $v$ is fixed as $v(\tau)=\theta(\gamma\tau)/\theta(\tau)$ and we shall drop the multiplier system and also the group  from the notation and write simply $M_k$ and $S_k$. In this case, we have the operators $U_4$ and $W_4$ as follows
\begin{align*}
f|U_4&=\frac{1}{4}\sum_{j\mod 4}f(\frac{\tau+j}{4}),\\
f|_kW_4&=(-2i\tau)^{-k}f(-\frac{1}{4\tau}).
\end{align*}
Explicitly, if $f=\sum_{n=0}^\infty c(n)q^n$, then $f|U_4=\sum_{n=0}^\infty c(4n)q^n$. If $c(n)=0$ whenever $(-1)^{k-\frac{1}{2}}n\equiv 2,3\mod 4$, $f$ is called to satisfy Kohnen's plus condition and the space of such cusp forms is denoted by $S_k^{+}$. Then $f\in S_k^+$ if and only if 
\[f|U_4=(-1)^\frac{2k-1}{4}2^kf|_kW_4.\]

\medskip

Let $k$ be a positive even integer and $\nu$ be a nonnegative integer. Let $\{f_1, \cdots, f_d\}$ (resp.  $\{F_1, \cdots, F_e\}$ ) be the orthogonal basis of normalized Hecke eigenforms of weight $k$ (resp. weight $2k+4\nu$) for $S_k$ (resp. $S_{2k+4\nu}$).  Let $\{g_1, \cdots, g_e \}$ be an orthogonal basis of Hecke eigenforms of 
$ S^+ _{k+2\nu+\frac{1}{2}} $ with $g_{\mu}=\sum_{n\geq 1}c_{\mu}(n)q^n$ corresponding to $F_{\mu}$ for all $\mu$ so that $g_{\mu}$ and $F_{\mu}$ have the same Hecke eigenvalues, so they correspond under Shimura's correspondence. Recall also that the first Shimura map
$$\mathcal{S}_1 : S^+_{k+2\nu+\frac{1}{2}}\rightarrow S_{2k+4\nu}$$ is given by
$$\sum_{n\geq 1} c(n)q^n \rightarrow \sum_{n\geq 1}\big(\sum_{d|n}d^{k-1}c(\frac{n^2}{d^2})\big)q^n.$$ 
and commutes with all Hecke operators.  One has
$\mathcal{S}_1(g_{\mu})=c_{\mu}(1)F_{\mu}$.

\medskip

For even integral weight $k$, we first extend the identity 
\[\mathcal{S}_1(\theta(\tau)f(4\tau))=f^2(4\tau)\]
for normalized Hecke eigenforms $f\in S_k$ to higher Rankin-Cohen brackets. For Eisenstein series, the next proposition was proved in \cite{KZ} when $\nu=0$ and in Prop B.1 of \cite{P} for general $\nu$. Note that the formula in \cite{P} has a wrong scalar. 

\begin{prop} For any even integral weight $k$, any normalized Hecke eigenform $f \in S_k$, and any nonnegative integer $\nu$,
$$\frac{(k+2\nu-1)!}{(k+\nu-1)!}\mathcal{S}_1\big([\theta(\tau),f(4\tau) ]_{\nu}\big)=[f(\tau), f(\tau)]_{2\nu}$$
\end{prop}

Next we generalize a formula of Zagier in \cite{Z} to Poincar\'e series and relax the assumption on multiplier systems. Let $\nu$ be a nonnegative integer, $k_1,k_2\in\mathbb{R}$ with $k_2>2$, $v_1,v_2$ multiplier systems for $\Gamma$, $v=v_1v_2$.
Assume
\begin{align*}
    g=\sum_{n\geq 0}b(n)e^{2\pi in\tau/w}&\in M_{k_1}(\Gamma,v_1)\\
    f=\sum_{n>0}a(n)e^{2\pi in\tau/w}&\in S_{k_1+k_2+2\nu}(\Gamma,v).
\end{align*}
Assume $v_i(T^w)=e^{2\pi i\delta_i}$ with $\delta_i\in[0,1)$. Then $b(n)=0$ if $n\not\in \mathbb Z+\delta_1$, and $a(n)=0$ if $n\not\in \mathbb Z+\delta_1+\delta_2$. Let $m$ be a positive real number in $\mathbb Z+\delta_2$.


\begin{prop} Let $\nu$ be a nonnegative integer and $m$ a positive real number in $\mathbb Z+\delta_2$.We have
\begin{align*}
&\left\langle f,[g,P_{m,k_2,v_2}]_\nu\right\rangle=\frac{w^{k_1+k_2+\nu}}{[\Gamma(1):\Gamma]}\frac{\Gamma(k_1+k_2+2\nu-1)}{(4\pi)^{k_1+k_2+2\nu-1}}\sum_{\mu=0}^\nu (-m)^{\mu}\genfrac{(}{)}{0pt}{}{\nu}{\mu}\frac{\Gamma(k_1+\nu)\Gamma(k_2+\nu)}{\Gamma(k_1+\nu-\mu)\Gamma(k_2+\mu)}\\&\qquad \times\sum_{n\in\mathbb Z+\delta_1} n^{\nu-\mu}\frac{a(m+n)\overline{b(n)}}{(m+n)^{k_1+k_2+2\nu-1}}. 
\end{align*}
\end{prop}

With the preceding proposition, we can prove an identity of the bracket $[\theta(\tau),P_{k,m}(4\tau)]$, generalizing Proposition of \cite{K}.

\begin{prop}
Let $\nu$ be a non-negative integer and $k\geq 4$ be an even integer. Let \[g=\sum_{n=1}^\infty c(n)q^n\in S_{k+2\nu+1/2}^+(\Gamma_0(4))\] and $P_{k,m}$ be the Poincar\'e series of weight $k$, index $m$ and level $1$. Then
\begin{align*}&\left\langle g,[\theta,P_{k,m}(4\tau)]_\nu\right\rangle\\=&\frac{\Gamma(k+2\nu-1/2)}{2^{2k+4\nu+1}\pi^{k+2\nu-1/2}}\sum_{\mu=0}^\nu (-4m)^{\mu}\genfrac{(}{)}{0pt}{}{\nu}{\mu}\frac{\Gamma(1/2+\nu)\Gamma(k+\nu)}{\Gamma(1/2+\nu-\mu)\Gamma(k+\mu)}\sum_{n\in \mathbb Z} \frac{n^{2\nu-2\mu}c(4m+n^2)}{(4m+n^2)^{k+2\nu-1/2}}.
\end{align*}
\end{prop}
 \medskip

Our main theorem connects the Rankin-Cohen brackets $[\theta(\tau),f_j(4\tau)]_\nu$ with the basis $\{g_\mu\}$. Note that we have employed the same normalizations for the Petersson inner product regardless whether the modular forms have integral or half-integral weights, while they are different in \cite{K}. Therefore, our formula differs from that of \cite{K} by a scalar of $6=[\Gamma(1):\Gamma_0(4)]$ in the case $\nu=0$.

\begin{thm} 
For each $m, \nu\in \mathbb{N}$, 
  $$ \sum_{j=1}^d\frac{a_{j}(m)}{m^{k-1}\langle f_{j}, f_{j}\rangle }
 [ \theta(\tau),f_{j}(4 \tau)]_{\nu}=
\sum_{\mu=1}^e \frac{\ell_{\nu}(g_\mu,m)}{\langle g_\mu, g_\mu\rangle} g_{\mu}(\tau)$$
where for $g=\sum_{n=1}^\infty c(n)q^n\in S_{k+2\nu+1/2}^+(\Gamma_0(4))$ 
\begin{align*}
 \ell_{\nu}(g,m)&=\frac{\Gamma(k+2\nu-\frac{1}{2})}{2^2  (4\pi)^{2\nu+1/2}\Gamma(k-1)} \\
&\qquad\times\sum_{\mu=0}^{\nu} (-4m)^\mu\sm \nu\\ \mu\esm 
\frac{\Gamma(1/2+\nu)\Gamma(k+\nu)}{\Gamma(1/2+\nu-\mu)\Gamma(k+\mu)} \sum_{n\in \mathbb{Z}}
\frac{n^{2\nu-2\mu}c(4m+n^2)}{(4m+n^2)^{k+2\nu-\frac{1}{2}}}.    
\end{align*}
\end{thm}

As a corollary, we apply the first Shimura lift and relate the brackets $[f_j,f_j]_\nu$ to the basis $\{F_\mu\}$.

\begin{cor}
For each $m, \nu\in \mathbb{N}$,
$$ \frac{\Gamma(k+\nu)}{\Gamma(k+2\nu)}\sum_{j=1}^d\frac{a_{j}(m)}{m^{k-1}\langle f_{j}, f_{j}\rangle }[f_j,f_{j}]_{2\nu}=
\sum_{\mu=1}^e \frac{\ell_{\nu}(g_\mu,m)}{\langle g_\mu, g_\mu\rangle} c_\mu(1)F_{\mu}$$
\end{cor}

These results are the analogues 
for greater $\nu$ of \cite{K} in the case $\nu=0$ and similar corollaries (special values, 
algebraicity etc.) as in \cite{K} follow easily. Note that there are a few misprints on the $2$-powers in \cite{K}: in the expression of $\textrm{tr } G_2$ on Page 295, the $2$-power should be $2^{-k+2}$, and hence the rational constant $\kappa_k$ should take the form
\[\kappa_k=\frac{3}{2}\cdot \frac{1\cdot 3\cdot 5\cdot \ldots \cdot (2k-3)}{2^{k}(k-2)!}.\]

\section{Proofs}

\subsection{Proof of Proposition 2.1} 
We verify the details for completeness, following the proof in \cite{P}.

Assume $f=\sum_{n=1}^\infty a(n)q^n\in S_k$, so $k\geq 12$ is even. Note first that 
\begin{align*}
&[\theta(\tau),f(4\tau)]_{\nu}\\
=&\sum_{\mu=0}^\nu(-1)^{\nu-\mu}\genfrac{(}{)}{0pt}{}{\nu}{\mu}\frac{\Gamma(1/2+\nu)\Gamma(k+\nu)}{\Gamma(1/2+\mu)\Gamma(k+\nu-\mu)}\left(\sum_{r\in\mathbb Z}r^{2\mu}q^{r^2}\right)\left(\sum_{n=1}^\infty a(n)(4n)^{2(\nu-\mu)}q^{4n}\right)\\ 
=&\sum_{\mu=0}^\nu(-1)^{\nu-\mu}\genfrac{(}{)}{0pt}{}{\nu}{\mu}\frac{\Gamma(1/2+\nu)\Gamma(k+\nu)}{\Gamma(1/2+\mu)\Gamma(k+\nu-\mu)}\sum_{n=1}^\infty\left(\sum_{r\in\mathbb Z}r^{2\mu}(n-r^2)^{\nu-\mu}a\left(\frac{n-r^2}{4}\right)\right)q^n,  
\end{align*}
where we understand that $a(x)=0$ if $x\not\in\mathbb N$. It follows by the definitio of $\mathcal{S}_1$ that the $n$-th Fourier coefficient of $\mathcal{S}_1\big([\theta(\tau),f(4\tau) ]_{\nu}\big)$ is given by
\begin{align*}
&\sum_{d\mid n}d^{k+2\nu-1}\sum_{\mu=0}^\nu(-1)^{\nu-\mu}\genfrac{(}{)}{0pt}{}{\nu}{\mu}\frac{\Gamma(1/2+\nu)\Gamma(k+\nu)}{\Gamma(1/2+\mu)\Gamma(k+\nu-\mu)}\sum_{r\in\mathbb Z}r^{2\mu}(n^2/d^2-r^2)^{\nu-\mu}a\left(\frac{n^2-d^2r^2}{4d^2}\right).
\end{align*}
By well-known formulas of $\Gamma$-function, $n$-th Fourier coefficient of $\frac{1}{(k+\nu-1)!}\mathcal{S}_1\big([\theta(\tau),f(4\tau) ]_{\nu}\big)$ is equal to
\begin{align*}
&\sum_{d\mid n}d^{k+2\nu-1}\sum_{\mu=0}^\nu\frac{(-1)^{\nu-\mu}(2\nu)!}{(2\mu)!(\nu-\mu)!(k+\nu-\mu-1)!}\sum_{r\in\mathbb Z}r^{2\mu}\left(\frac{n^2-d^2r^2}{4d^2}\right)^{\nu-\mu}a\left(\frac{n^2-d^2r^2}{4d^2}\right).
\end{align*}
Now the $x^{2\nu}$-coefficient of the following polynomial
\[\left(\frac{n^2-r^2d^2}{4d^2}x^2+rx-1\right)^{k+2\nu-1}=\left(\frac{n+rd}{2d}x-1\right)^{k+2\nu-1}\left(\frac{n-rd}{2d}x+1\right)^{k+2\nu-1}\]
is given by 
\begin{align}
&\sum_{\mu=0}^\nu\genfrac{(}{)}{0pt}{}{k+2\nu-1}{\nu-\mu}\genfrac{(}{)}{0pt}{}{k+\nu+\mu-1}{2\mu}r^{2\mu}(-1)^{\nu-\mu+1}\left(\frac{n^2-r^2d^2}{4d^2}\right)^{\nu-\mu}\nonumber\\
=&\sum_{\mu=0}^\nu\frac{(k+2\nu-1)!}{(2\mu)!(\nu-\mu)!(k+\nu-\mu-1)!}r^{2\mu}(-1)^{\nu-\mu+1}\left(\frac{n^2-r^2d^2}{4d^2}\right)^{\nu-\mu}
\end{align}
and also by
\begin{align}
&\sum_{\mu=0}^{2\nu}\genfrac{(}{)}{0pt}{}{k+2\nu-1}{\mu}\genfrac{(}{)}{0pt}{}{k+2\nu-1}{2\nu-\mu}(-1)^{\mu+1}\left(\frac{n+rd}{2d}\right)^{\mu}\left(\frac{n-rd}{2d}\right)^{2\nu-\mu}.
\end{align}
By (2.1) and (2.2), the $n$-th Fourier coefficient of $\frac{1}{(k+\nu-1)!}\mathcal{S}_1\big([\theta(\tau),f(4\tau) ]_{\nu}\big)$ is equal to
\begin{align*}
&\frac{-(2\nu)!}{(k+2\nu-1)!}\sum_{d\mid n}d^{k+2\nu-1}\sum_{\mu=0}^{2\nu}\genfrac{(}{)}{0pt}{}{k+2\nu-1}{\mu}\genfrac{(}{)}{0pt}{}{k+2\nu-1}{2\nu-\mu}\\&\qquad\times\sum_{r\in\mathbb Z}(-1)^{\mu+1}\left(\frac{n+rd}{2d}\right)^{\mu}\left(\frac{n-rd}{2d}\right)^{2\nu-\mu}a\left(\frac{n^2-d^2r^2}{4d^2}\right).
\end{align*}
Now by the change of variables $n_1=\frac{n+dr}{2}$ and $n_2=\frac{n-dr}{2}$, this is equal to
\begin{align*}
&\frac{-(2\nu)!}{(k+2\nu-1)!}\sum_{n_1+n_2=n}\sum_{d\mid (n_1,n_2)}d^{k-1}\sum_{\mu=0}^{2\nu}\genfrac{(}{)}{0pt}{}{k+2\nu-1}{\mu}\genfrac{(}{)}{0pt}{}{k+2\nu-1}{2\nu-\mu}\\&\qquad\times(-1)^{\mu+1}n_1^{\mu}n_2^{2\nu-\mu}a\left(\frac{n_1n_2}{d^2}\right),
\end{align*}
since $\frac{n}{d}\equiv r\mod 2$ and hence $d\mid n$ if and only if $d\mid n_1$ and $d\mid n_2$. Recall the Hecke relation of $a(n)$:
\[a(n)a(m)=\sum_{d\mid(m,n)}d^{k-1}a(\frac{mn}{d^2}),\]
so the $n$-th Fourier coefficient of the left-hand side is equal to
\begin{align*}
&\sum_{n_1+n_2=n}\sum_{\mu=0}^{2\nu}(2\nu)!\genfrac{(}{)}{0pt}{}{k+2\nu-1}{\mu}\genfrac{(}{)}{0pt}{}{k+2\nu-1}{2\nu-\mu}(-1)^{\mu}n_1^{\mu}n_2^{2\nu-\mu}a(n_1)a(n_2).
\end{align*}

On the other hand, the right-hand side of the desired formula is equal to 
\begin{align*}
&[f(\tau), f(\tau)]_{2\nu}\\
=&\sum_{\mu=0}^{2\nu}(-1)^{2\nu-\mu}\genfrac{(}{)}{0pt}{}{2\nu}{\mu}\frac{\Gamma(k+2\nu)\Gamma(k+2\nu)}{\Gamma(k+\mu)\Gamma(k+2\nu-\mu)}\left(\sum_{n=1}^\infty a(n)n^{\mu}q^{n}\right)\left(\sum_{n=1}^\infty a(n)n^{2\nu-\mu}q^{n}\right)\\ 
=&\sum_{\mu=0}^{2\nu}(-1)^{\mu}\genfrac{(}{)}{0pt}{}{2\nu}{\mu}\frac{\Gamma(k+2\nu)\Gamma(k+2\nu)}{\Gamma(k+\mu)\Gamma(k+2\nu-\mu)}\sum_{n=1}^\infty \left(\sum_{n_1+n_2=n}a(n_1)a(n_2)n_1^\mu n_2^{2\nu-\mu}\right)q^n, 
\end{align*}
so the $n$-th Fourier coefficient of the right-hand side is given by
\[\sum_{\mu=0}^{2\nu}(-1)^{\mu}(2\nu)!\genfrac{(}{)}{0pt}{}{k+2\nu-1}{\mu}\genfrac{(}{)}{0pt}{}{k+2\nu-1}{2\nu-\mu} \sum_{n_1+n_2=n}a(n_1)a(n_2)n_1^\mu n_2^{2\nu-\mu}.\]
And we complete the proof.
  
\subsection{Proof of Proposition 2.2}

We extend Zagier's formula to Poincar\'e series and relax his assumption on the multiplier systems. 

\begin{proof}
By definition,
\begin{align*}
[g,P_{m,k_2,v_2}]_\nu&=(2\pi i)^{-\nu}\sum_{\mu=0}^\nu(-1)^{\nu-\mu}\genfrac{(}{)}{0pt}{}{\nu}{\mu}\frac{\Gamma(k_1+\nu)\Gamma(k_2+\nu)}{\Gamma(k_1+\mu)\Gamma(k_2+\nu-\mu)}g^{(\mu)}(\tau) P_{m,k_2,\nu_2}^{(\nu-\mu)}(\tau)\\
&=(2\pi i)^{-\nu}\sum_{\mu=0}^\nu(-1)^{\nu-\mu}\genfrac{(}{)}{0pt}{}{\nu}{\mu}\frac{\Gamma(k_1+\nu)\Gamma(k_2+\nu)}{\Gamma(k_1+\mu)\Gamma(k_2+\nu-\mu)}g^{(\mu)}(\tau) \\
&\qquad \times\sum_{\gamma\in\Gamma_\infty\backslash\Gamma}v_2(\gamma)^{-1}\left((c\tau+d)^{-k_2}e^{2\pi im\gamma\tau/w}\right)^{(\nu-\mu)}.
\end{align*}
By induction, we see that
\begin{align*}
&\left((c\tau+d)^{-k_2}e^{2\pi im\gamma\tau/w}\right)^{(\nu-\mu)}\\=&\sum_{r=0}^{\nu-\mu}\genfrac{(}{)}{0pt}{}{\nu-\mu}{r}\frac{\Gamma(k_2+\nu-\mu)}{\Gamma(k_2+r)}(-c)^{\nu-\mu-r}(2\pi imw^{-1})^r(c\tau+d)^{-k_2-(\nu-\mu)-r}e^{2\pi im\gamma\tau/w}.    
\end{align*}
Putting together, we obtain
\begin{align*}
&[g,P_{m,k_2,v_2}]_\nu\\
=&(2\pi i)^{-\nu}\sum_{\mu=0}^\nu(-1)^{\nu-\mu}\genfrac{(}{)}{0pt}{}{\nu}{\mu}\frac{\Gamma(k_1+\nu)\Gamma(k_2+\nu)}{\Gamma(k_1+\mu)\Gamma(k_2+\nu-\mu)}g^{(\mu)} (\tau)\sum_{\gamma\in\Gamma_\infty\backslash\Gamma}v_2(\gamma)^{-1}\\
&\qquad \times\sum_{r=0}^{\nu-\mu}\genfrac{(}{)}{0pt}{}{\nu-\mu}{r}\frac{\Gamma(k_2+\nu-\mu)}{\Gamma(k_2+r)}(-c)^{\nu-\mu-r}(2\pi imw^{-1})^r(c\tau+d)^{-k_2-(\nu-\mu)-r}e^{2\pi im\gamma\tau/w}.   
\end{align*}
Interchanging $r$ and $\mu$, this equals
\begin{align*}
&(2\pi i)^{-\nu}\sum_{\gamma\in\Gamma_\infty\backslash\Gamma}v_2(\gamma)^{-1}\sum_{\mu=0}^\nu\sum_{r=0}^{\nu-\mu}(-1)^{\nu-r}\genfrac{(}{)}{0pt}{}{\nu}{r}\frac{\Gamma(k_1+\nu)\Gamma(k_2+\nu)}{\Gamma(k_1+r)\Gamma(k_2+\nu-r)}g^{(r)}(\tau) \\
&\qquad \times\genfrac{(}{)}{0pt}{}{\nu-r}{\mu}\frac{\Gamma(k_2+\nu-r)}{\Gamma(k_2+\mu)}(-c)^{\nu-\mu-r}(2\pi imw^{-1})^\mu(c\tau+d)^{-k_2-(\nu-r)-\mu}e^{2\pi im\gamma\tau/w}\\
=&(2\pi i)^{-\nu}\sum_{\gamma\in\Gamma_\infty\backslash\Gamma}v(\gamma)^{-1}\sum_{\mu=0}^\nu(2\pi imw^{-1})^\mu(-1)^{\mu}\genfrac{(}{)}{0pt}{}{\nu}{\mu}\frac{\Gamma(k_1+\nu)\Gamma(k_2+\nu)}{\Gamma(k_2+\mu)\Gamma(k_1+\nu-\mu)}(c\tau+d)^{-k_1-k_2-2\nu}\\
&\qquad \times e^{2\pi im\gamma\tau/w}\sum_{r=0}^{\nu-\mu} v_1(\gamma)g^{(r)}(\tau) \genfrac{(}{)}{0pt}{}{\nu-\mu}{r}\frac{\Gamma(k_1+\nu-\mu)}{\Gamma(k_1+r)}c^{\nu-\mu-r}(c\tau+d)^{k_1+\nu-\mu+r},   
\end{align*}
where we used the identity 
\[\genfrac{(}{)}{0pt}{}{\nu}{\mu}\genfrac{(}{)}{0pt}{}{\nu-\mu}{r}=\genfrac{(}{)}{0pt}{}{\nu}{r}\genfrac{(}{)}{0pt}{}{\nu-r}{\mu}.\]
By Zagier's equation (76), this equals
\begin{align*}
&(2\pi i)^{-\nu}\sum_{\mu=0}^\nu(2\pi imw^{-1})^\mu(-1)^{\mu}\genfrac{(}{)}{0pt}{}{\nu}{\mu}\frac{\Gamma(k_1+\nu)\Gamma(k_2+\nu)}{\Gamma(k_2+\mu)\Gamma(k_1+\nu-\mu)}\\
&\qquad \times \sum_{\gamma\in\Gamma_\infty\backslash\Gamma}v(\gamma)^{-1}(c\tau+d)^{-k_1-k_2-2\nu}e^{2\pi im\gamma\tau/w}g^{(\nu-\mu)}(\gamma \tau)\\
=& (2\pi i)^{-\nu}\sum_{\mu=0}^\nu(2\pi imw^{-1})^\mu(-1)^{\mu}\genfrac{(}{)}{0pt}{}{\nu}{\mu}\frac{\Gamma(k_1+\nu)\Gamma(k_2+\nu)}{\Gamma(k_2+\mu)\Gamma(k_1+\nu-\mu)}\\
&\qquad\times\sum_{\gamma\in\Gamma_\infty\backslash\Gamma}\left.\left(e^{2\pi im\tau/w}g^{(\nu-\mu)}(\tau)\right)\right|_{k_1+k_2+2\nu,v}\gamma(\tau).
\end{align*}
Now $g^{(\nu-\mu)}(\tau)=(2\pi iw^{-1})^{\nu-\mu}\sum_{n\in\mathbb Z+\delta_1}n^{\nu-\mu}b(n)e^{2\pi in\tau/w}$, so
\begin{align*}
&(2\pi i)^{-\nu}\sum_{\mu=0}^\nu(2\pi imw^{-1})^\mu(-1)^{\mu}\genfrac{(}{)}{0pt}{}{\nu}{\mu}\frac{\Gamma(k_1+\nu)\Gamma(k_2+\nu)}{\Gamma(k_2+\mu)\Gamma(k_1+\nu-\mu)}e^{2\pi im\tau/w}g^{(\nu-\mu)}(\tau)\\
=& (2\pi i)^{-\nu}\sum_{\mu=0}^\nu(2\pi imw^{-1})^\mu(-1)^{\mu}\genfrac{(}{)}{0pt}{}{\nu}{\mu}\frac{\Gamma(k_1+\nu)\Gamma(k_2+\nu)}{\Gamma(k_2+\mu)\Gamma(k_1+\nu-\mu)}\\
&\qquad \times e^{2\pi im\tau/w}(2\pi iw^{-1})^{\nu-\mu}\sum_{n\in\mathbb Z+\delta_1} n^{\nu-\mu}b(n)e^{2\pi in\tau/w}\\
=& w^{-\nu}\sum_{\mu=0}^\nu (-m)^{\mu}\genfrac{(}{)}{0pt}{}{\nu}{\mu}\frac{\Gamma(k_1+\nu)\Gamma(k_2+\nu)}{\Gamma(k_2+\mu)\Gamma(k_1+\nu-\mu)}\sum_{n\in\mathbb Z+\delta_1} n^{\nu-\mu}b(n)e^{2\pi i(n+m)\tau/w}.
\end{align*}
It follows that
\begin{align*}
&[g,P_{m,k_2,v_2}]_\nu\\
=& w^{-\nu}\sum_{\mu=0}^\nu (-m)^{\mu}\genfrac{(}{)}{0pt}{}{\nu}{\mu}\frac{\Gamma(k_1+\nu)\Gamma(k_2+\nu)}{\Gamma(k_2+\mu)\Gamma(k_1+\nu-\mu)}\sum_{n\in\mathbb Z+\delta_1} n^{\nu-\mu}b(n)P_{n+m,k_1+k_2+2\nu,v}(\tau),
\end{align*}
which implies that
\begin{align*}
\left\langle f,[g,P_{m,k_2,v_2}]_\nu\right\rangle=&\frac{w^{k_1+k_2+\nu}}{[\Gamma(1):\Gamma]}\sum_{\mu=0}^\nu (-m)^{\mu}\genfrac{(}{)}{0pt}{}{\nu}{\mu}\frac{\Gamma(k_1+\nu)\Gamma(k_2+\nu)}{\Gamma(k_2+\mu)\Gamma(k_1+\nu-\mu)}\\
&\qquad\times\sum_{n\in\mathbb Z+\delta_1} n^{\nu-\mu}\overline{b(n)}\frac{\Gamma(k_1+k_2+2\nu-1)}{(4\pi(m+n))^{k+2\nu-1}}a(m+n).    
\end{align*}
This is the desired formula.
\end{proof}

\subsection{Proof of Proposition 2.3}

In general, Rankin-Cohen brackets commute with slash-$k$ operators by Cohen.

\begin{lem}[Cohen 1975]
Let $g,h$ be modular forms of level $4$ with weight $k_1\in \mathbb Z$ and $k_2\in\frac{1}{2}+\mathbb Z$ respectively. Then
\[[g|_{k_1}W_4,h|_{k_2}W_4]_\nu=i^{-k_1-2\nu}[g,h]_\nu|_{k_1+k_2+2\nu}W_4,\]
where 
\begin{align*}
g|_{k_1}W_4&=\det(W_4)^{\frac{k_1}{2}}j(W_4,\tau)^{-k_1}g(W_4\tau)\\
h|_{k_2}W_4&=(-2i\tau)^{-k_2}h(W_4\tau)
\end{align*}
\end{lem}
\begin{proof}
Cohen showed that the bracket commutes with slash-$\alpha$ for any $\alpha\in\textrm{SL}_2(\mathbb R)$.
Now in our setting, let $\alpha=\begin{pmatrix}0&-1/2\\2&0\end{pmatrix}$ and we see easily
\begin{align*}
g|_{k_1}W_4&=g|_{k_1}\alpha\\
h|_{k_2}W_4&=i^{k_2}h|_{k_2}\alpha\\
[g,h]_\nu|_{k_1+k_2+2\nu}W_4&=i^{k_1+k_2+2\nu}[g,h]_\nu|_{k_1+k_2+2\nu}\alpha.
\end{align*}
It follows that
\begin{align*}
[g|_{k_1}W_4,h|_{k_2}W_4]_\nu&=i^{k_2}[g|_{k_1}\alpha,h|_{k_2}\alpha]_\nu\\
&=i^{k_2}[g,h]_\nu|_{k_1+k_2+2\nu}\alpha\\
&=i^{k_2}i^{-k_1-k_2-2\nu}[g,h]_\nu|_{k_1+k_2+2\nu}W_4=i^{-k_1-2\nu}[g,h]_\nu|_{k_1+k_2+2\nu}W_4.
\end{align*}
This is the formula we wanted.
\end{proof}

\begin{lem}
Let $g\in M_{k_1}(\Gamma(1))$, $h\in M_{k_2}^+(\Gamma_0(4))$ with $k_1\in \mathbb Z$ and $k_2\in\frac{1}{2}+\mathbb Z$, and $f\in S_{k}^+(\Gamma_0(4))$ with $k=k_1+k_2+2\nu$. Then
\[\langle f,[h|U_4,g(4\tau)]_\nu\rangle =\genfrac{(}{)}{0.5pt}{}{2}{2k_2}2^{-k_1+k_2-1/2}i^{k_1+2\nu}\langle f|_{k_1+k_2+2\nu}W_4,[h,g]_\nu\rangle.\]
\end{lem}
\begin{proof}
Since $g$ is of level one and $W_4=SV_4$, we see that
\[g|_{k_1}W_4=g|_{k_1}SV_4=g|_{k_1}V_4=2^{k_1}g(4\tau).\]
On the other hand,
$h$ belongs to the plus space, so $h$ is an eigenform for $U_4W_4$ with eigenvalue $\genfrac{(}{)}{0.5pt}{}{2}{2k_2}2^{k_2-1/2}$, where
\[h|U_4=\frac{1}{4}\sum_{j\mod 4}h(\frac{\tau+j}{4}).\]
Then we have
\[[h|_{k_2}W_4,g|_{k_1}W_4]_\nu=\genfrac{(}{)}{0.5pt}{}{2}{2k_2}2^{k_1-k_2+1/2}[h|U_4(\tau),g(4\tau)]_\nu.\]
Therefore, by the preceding lemma, we have
\begin{align*}
\langle f,[h|U_4,g(4\tau))]_\nu\rangle &=\genfrac{(}{)}{0.5pt}{}{2}{2k_2}2^{-k_1+k_2-1/2}\langle f,[h|_{k_2}W_4,g|_{k_1}W_4]_\nu\rangle \\
&=\genfrac{(}{)}{0.5pt}{}{2}{2k_2}2^{-k_1+k_2-1/2}i^{k_1+2\nu}\langle f,[h,g]_\nu|_{k_1+k_2+2\nu}W_4\rangle \\
&=\genfrac{(}{)}{0.5pt}{}{2}{2k_2}2^{-k_1+k_2-1/2}i^{k_1+2\nu}\langle f|_{k_1+k_2+2\nu}W_4,[h,g]_\nu\rangle,
\end{align*}
noting that $W_4$ is a unitary involution.
\end{proof}

\begin{prop}
Let $g\in M_{k_1}(\Gamma(1))$ and $f\in S_{k_1+1/2+2\nu}^+(\Gamma_0(4))$. Then
\[\langle f,[\theta,g(4\tau))]_\nu\rangle =2^{-2k_1-2\nu}\langle f|U_4,[\theta,g]_\nu\rangle.\]
\end{prop}
\begin{proof}
In this case, $k_2=1/2$ and $\theta|U_4=\theta$. Moreover,
\[f|_{k_1+1/2+2\nu}W_4=(-1)^{\frac{k_1}{2}+\nu}2^{-k_1-2\nu}f|U_4.\]
The formula in the preceding lemma simplifies to the one in the statement.
\end{proof}

\medskip
Finally we can give a proof of Proposition 2.3.
\medskip

\noindent{\it Proof of Proposition 2.3}. By the Proposition 3.3 with $g=P_{k,m}$, we have
\[\left\langle f,[\theta,P_{k,m}(4\tau)]_\nu\right\rangle=2^{-2k-2\nu}\left\langle f|U_4,[\theta,P_{k,m}]_\nu\right\rangle.\]
Since $P_{k,m}$ is of level one, in order to apply Zagier's formula, we replace it with
\[P_{k,m}=P_{k,m,4}|\textrm{tr}=\sum_{\gamma\in\Gamma_0(4)\backslash\Gamma(1)}P_{k,m,4}|_k\gamma.\]
We then choose the representatives for $\gamma$ (see page 290 of \cite{K})
\[\begin{pmatrix}
1&0\\4&1    
\end{pmatrix},\quad \begin{pmatrix}
1&0\\2&1    
\end{pmatrix},\quad \begin{pmatrix}
1&0\\1&1    
\end{pmatrix}\begin{pmatrix}
1&j\\0&1    
\end{pmatrix}, j\mod 4.\]
For each $\gamma$, we have
\begin{align*}
\left\langle f|U_4,[\theta,P_{k,m,4}|_{k}\gamma]_\nu\right\rangle&=\left\langle f|U_4,[\theta|_{1/2}\gamma^{-1},P_{k,m,4}]_\nu|_{k+2\nu+1/2}\gamma\right\rangle\\
&=\left\langle f|_{k+2\nu+1/2}U_4\gamma^{-1},[\theta|_{1/2}\gamma^{-1},P_{k,m,4}]_\nu\right\rangle.
\end{align*}
We need to compute the Fourier expansion of 
\[f|_{k+2\nu+1/2}U_4\gamma^{-1},\quad \text{ and } \theta|_{1/2}\gamma^{-1}=\theta|_{1/2}U_4\gamma^{-1}\]
for each $\gamma$ and apply the corresponding Zagier's formula. We compute the first and then specializes to the second.

\noindent\textbf{Case 1:} $\gamma=\begin{pmatrix}
1&0\\4&1    
\end{pmatrix}$ and $\gamma^{-1}=\begin{pmatrix}
1&0\\-4&1    
\end{pmatrix}\in\Gamma_0(4)$. Since $f|U_4$ is modular of level four, $f|_{k+2\nu+1/2}U_4\gamma^{-1}=f|U_4$ and hence $\theta|_{1/2}U_4\gamma^{-1}=\theta$. Therefore, their Fourier expansions are
\begin{align*}
f|_{k+2\nu+1/2}U_4\gamma^{-1}&=\sum_{n=1}^\infty c(4n)q^n\\
\theta|_{1/2}U_4\gamma^{-1}&=\sum_{n\in\mathbb Z}q^{n^2}.
\end{align*}
Therefore, the contribution of this case is given by 
\begin{align*}
&2^{-2k-2\nu}\left\langle f|_{k+2\nu+1/2}U_4\gamma^{-1},[\theta|_{1/2}\gamma^{-1},P_{k,m,4}]_\nu\right\rangle\\=&\frac{\Gamma(k+2\nu-1/2)}{6\cdot 2^{4k+6\nu-1}\pi^{k+2\nu-1/2}}\sum_{\mu=0}^\nu (-m)^{\mu}\genfrac{(}{)}{0pt}{}{\nu}{\mu}\frac{\Gamma(1/2+\nu)\Gamma(k+\nu)}{\Gamma(1/2+\nu-\mu)\Gamma(k+\mu)}\sum_{n\in 2\mathbb Z} \frac{(n/2)^{2\nu-2\mu}c(4m+n^2)}{(m+n^2/4)^{k+2\nu-1/2}}\\=&\frac{\Gamma(k+2\nu-1/2)}{6\cdot 2^{2k+4\nu}\pi^{k+2\nu-1/2}}\sum_{\mu=0}^\nu (-4m)^{\mu}\genfrac{(}{)}{0pt}{}{\nu}{\mu}\frac{\Gamma(1/2+\nu)\Gamma(k+\nu)}{\Gamma(1/2+\nu-\mu)\Gamma(k+\mu)}\sum_{n\in 2\mathbb Z} \frac{n^{2\nu-2\mu}c(4m+n^2)}{(4m+n^2)^{k+2\nu-1/2}}, 
\end{align*}
since $[\Gamma(1):\Gamma_0(4)]=6$.

\noindent\textbf{Case 2:} $\gamma=\begin{pmatrix}
1&0\\2&1    
\end{pmatrix}$ and $\gamma^{-1}=\begin{pmatrix}
1&0\\-2&1    
\end{pmatrix}=W_4^{-1}\begin{pmatrix}
1&1/2\\0&1    
\end{pmatrix}W_4$. Since $f$ has weight $k+2\nu+1/2$ and lies in the plus space,
\[f|U_4W_4^{-1}=(-1)^{k/2+\nu}2^{k+2\nu}f.\]
Therefore,
\begin{align*}
f|_{k+2\nu+1/2}U_4\gamma^{-1}&=(-1)^{k/2+\nu}2^{k+2\nu}f|_{k+2\nu+1/2}\begin{pmatrix}
1&1/2\\0&1    
\end{pmatrix}W_4.
\end{align*}
Now \[f(\tau+1/2)=(-1)^{k/2+\nu}2^{k+2\nu+1}(f|_{k+2\nu+1/2}W_4)(4\tau)-f(\tau)=(-1)^{k/2+\nu}2^{1/2}f|_{k+2\nu+1/2}W_4V_4-f\] so
\begin{align*}
f|_{k+2\nu+1/2}U_4\gamma^{-1}&=2^{k+2\nu+1/2}f|_{k+2\nu+1/2}W_4V_4W_4-(-1)^{k/2+\nu}2^{k+2\nu}f|W_4\\
&=f(\tau/4)-f|U_4(\tau)=\sum_{n=1}^\infty c(n)q^{\frac{n}{4}}-\sum_{n=1}^\infty c(4n)q^n\\
&=\sum_{n\geq 1, \text{ odd }} c(n)q^{\frac{n}{4}}.
\end{align*}
Specializing to $\theta$, we have
\[\theta|_{1/2}U_4\gamma^{-1}=\sum_{n\in 1+2\mathbb Z}q^{\frac{n^2}{4}}.\]
Note that the multiplier systems for $f|_{k+2\nu+1/2}U_4\gamma^{-1}$ and $\theta|_{1/2}U_4\gamma^{-1}$ are the same, both being given by $v^{\gamma^{-1}}(\alpha)=v(\gamma^{-1}\alpha\gamma)$ with $v$ that of $\theta$. In particular, $v^{\gamma^{-1}}(T)=i$ and by Proposition 1.1 the contribution of this case is given by:
\begin{align*}&2^{-2k-2\nu}\left\langle f|_{k+2\nu+1/2}U_4\gamma^{-1},[\theta|_{1/2}\gamma^{-1},P_{k,m,4}]_\nu\right\rangle\\=&\frac{\Gamma(k+2\nu-1/2)}{6\cdot 2^{4k+6\nu-1}\pi^{k+2\nu-1/2}}\sum_{\mu=0}^\nu (-m)^{\mu}\genfrac{(}{)}{0pt}{}{\nu}{\mu}\frac{\Gamma(1/2+\nu)\Gamma(k+\nu)}{\Gamma(1/2+\nu-\mu)\Gamma(k+\mu)}\sum_{n\in 1+2\mathbb Z} \frac{(n/2)^{2\nu-2\mu}c(4m+n^2)}{(m+\frac{n^2}{4})^{k+2\nu-1/2}}\\=&\frac{\Gamma(k+2\nu-1/2)}{6\cdot 2^{2k+4\nu}\pi^{k+2\nu-1/2}}\sum_{\mu=0}^\nu (-4m)^{\mu}\genfrac{(}{)}{0pt}{}{\nu}{\mu}\frac{\Gamma(1/2+\nu)\Gamma(k+\nu)}{\Gamma(1/2+\nu-\mu)\Gamma(k+\mu)}\sum_{n\in 1+2\mathbb Z} \frac{n^{2\nu-2\mu}c(4m+n^2)}{(4m+n^2)^{k+2\nu-1/2}}.
\end{align*}

\noindent\textbf{Case 3:} $\gamma_j=\begin{pmatrix}
1&0\\1&1    
\end{pmatrix}\begin{pmatrix}
1&j\\0&1    
\end{pmatrix}$, $j\mod 4$, so $\gamma_j^{-1}=\begin{pmatrix}
1&-j\\0&1    
\end{pmatrix}\begin{pmatrix}
1&0\\-1&1    
\end{pmatrix}$. But $f|U_4$ is modular for $\Gamma_0(4)$, so different $\gamma_j$'s give the same Fourier expansion:
\[f|_{k+2\nu+1/2}U_4\gamma_j^{-1}=f|_{k+2\nu+1/2}U_4\gamma_0^{-1}=f|_{k+2\nu+1/2}U_4\begin{pmatrix}
1&0\\-1&1    
\end{pmatrix}.\]
Following Winfried, consider $\Gamma_0^0(4)$ and $f|_{k+2\nu+1/2}U_4\gamma_0^{-1}$ is a modular form for $\Gamma_0^0(4)$ whose multiplier system on $T^4$ has value $1$. A complete set of coset representatives for $\langle T^4\rangle\backslash \Gamma_0^0(4)$ is also a complete set of coset representatives for $\langle T\rangle\backslash \Gamma_0(4)$, so it is trivial that $P_{k,m,4}=P_{k,4m,\Gamma_0^0(4)}$, where the right-hand side is the $4m$-th Poincar\'e series for $\Gamma_0^0(4)$ of weight $k$. 

Now we compute the Fourier coefficients of $f|_{k+2\nu+1/2}U_4\gamma_0^{-1}$. Note that in the metaplectic double cover $\textrm{Mp}_2(\mathbb Z)$ of $\textrm{SL}_2(\mathbb Z)$ we have $\gamma_0^{-1}=T^{-1}S^{-1}T^{-1}$, so
\begin{align*}
f|_{k+2\nu+1/2}U_4\gamma_0^{-1}&=f|_{k+2\nu+1/2}U_4S^{-1}T^{-1}=(2i)^{k+2\nu}f|_{k+2\nu+1/2}W_4S^{-1}T^{-1}\\
&=(2i)^{k+2\nu}(-2i\tau)^{-(k+2\nu+1/2)}f(W_4\tau)|_{k+2\nu+1/2}S^{-1}T^{-1}\\
&=(2i)^{k+2\nu}(2i/\tau)^{-(k+2\nu+1/2)}(-\tau)^{-(k+2\nu+1/2)}f(\tau/4)|_{k+2\nu+1/2}T^{-1}\\
&=2^{-\frac{1}{2}}i^{k+2\nu+1/2}f(\frac{\tau-1}{4}).
\end{align*}
Plugging in the Fourier expansion of $f$, we see that
\begin{align*}
f|_{k+2\nu+1/2}U_4\gamma_0^{-1}&=2^{-\frac{1}{2}}i^{k+2\nu+1/2}\sum_{n=1}^\infty c(n)e^{-2\pi in/4}q^{\frac{n}{4}},
\end{align*}
and specializing to $f=\theta$ we obtain that
\begin{align*}
\theta|_{k+2\nu+1/2}U_4\gamma_0^{-1}&=2^{-\frac{1}{2}}i^{k+2\nu+1/2}\sum_{n\in\mathbb Z} e^{-2\pi in^2/4}q^{\frac{n^2}{4}}.
\end{align*}

Apply Proposition 1.1 to the group $\Gamma_0^0(4)$ and the total contribution in this case is given by
\begin{align*}&4\cdot 2^{-2k-2\nu}\left\langle f|_{k+2\nu+1/2}U_4\gamma_0^{-1},[\theta|_{1/2}\gamma_0^{-1},P_{k,m,4}]_\nu\right\rangle\\=&4\cdot 2^{-2k-2\nu}\left\langle f|_{k+2\nu+1/2}U_4\gamma_0^{-1},[\theta|_{1/2}\gamma_0^{-1},P_{k,4m,\Gamma_0^0(4)}]_\nu\right\rangle\\
=&\frac{4}{6\cdot 4}2^{-2k-2\nu}4^{k+1/2+\nu}\frac{\Gamma(k+2\nu-1/2)}{(4\pi)^{k+2\nu-1/2}}\sum_{\mu=0}^\nu (-4m)^{\mu}\genfrac{(}{)}{0pt}{}{\nu}{\mu}\frac{\Gamma(1/2+\nu)\Gamma(k+\nu)}{\Gamma(1/2+\nu-\mu)\Gamma(k+\mu)}\\&\qquad \times\sum_{n\in\mathbb Z} n^{2\nu-2\mu}2^{-1}\frac{c(4m+n^2)}{(4m+n^2)^{k+2\nu-1/2}}\\
=&\frac{1}{3}\frac{\Gamma(k+2\nu-1/2)}{2^{2k+4\nu}\pi^{k+2\nu-1/2}}\sum_{\mu=0}^\nu (-4m)^{\mu}\genfrac{(}{)}{0pt}{}{\nu}{\mu}\frac{\Gamma(1/2+\nu)\Gamma(k+\nu)}{\Gamma(1/2+\nu-\mu)\Gamma(k+\mu)}\sum_{n\in \mathbb Z} \frac{n^{2\nu-2\mu}c(4m+n^2)}{(4m+n^2)^{k+2\nu-1/2}}.
\end{align*}
Putting together the contributions of all cases, we have the desired formula.

\subsection{Proof of the main theorem and its corollary} Apply the relation
\[P_{k,m}=\frac{\Gamma(k-1)}{(4\pi m)^{k-1}}\sum_{j}\frac{a_j(m)}{\langle f_j,f_j\rangle}f_j\]
to Proposition 2.3 and we have Theorem 2.4. Then apply the first Shimura map and we obtain the identity in Corollary 2.5.



\bibliographystyle{amsplain}

\end{document}